\documentclass[11pt]{article}

\usepackage[margin=1.05in]{geometry}
\usepackage{amsmath,amssymb,amsthm,mathtools}
\usepackage{microtype}
\usepackage{enumitem}
\usepackage[hidelinks]{hyperref}
\usepackage[nameinlink,capitalize]{cleveref}

\newtheorem{theorem}{Theorem}[section]
\newtheorem{proposition}[theorem]{Proposition}
\newtheorem{lemma}[theorem]{Lemma}
\newtheorem{corollary}[theorem]{Corollary}
\newtheorem{remark}[theorem]{Remark}

\newcommand{\F}{\mathbb F}

\newcommand{\1}{\mathbf 1}

\newcommand{\eps}{\varepsilon}
\newcommand{\abs}[1]{\lvert #1\rvert}

\title{An Improved Exponent for Products of Differences over Arbitrary Finite Fields}
\author{Yao Zhi\thanks{E-mail: \texttt{ericyao2026@gmail.com}}}
\date{}

\begin{document}
\maketitle

\begin{abstract}
Murphy and Petridis proved that, for subsets $A$ of an arbitrary finite field $\F_q$, the condition $\abs A>q^{2/3-\delta}$ implies $\abs{(A-A)(A-A)}>q/2$ whenever $\delta<1/13542$ and $q$ is sufficiently large. We improve the explicit saving to $1/1689$. The proof replaces the main physical-space extraction in their large-energy branch by a common Fourier spectrum, followed by a popular-ratio Balog--Szemer\'edi--Gowers step. The resulting additive energy is treated level by level, retaining difference multiplicities and combining the Reiher--Schoen relation lemma with Mohammadi's difference--ratio estimate. This gives the structural exponent $188$ and hence
\[
\abs A\ge Cq^{375/563}=Cq^{2/3-1/1689}
\quad\Longrightarrow\quad
\abs{(A-A)(A-A)}>\frac q2
\]
for an absolute constant $C>0$.
\end{abstract}

\section{Introduction}

For $A\subseteq\F_q$, write
\[
(A-A)(A-A)=\{(a-b)(c-d):a,b,c,d\in A\}.
\]
This set is a natural four-variable manifestation of the finite-field sum--product phenomenon: additive structure enters through the two differences, while multiplication forces these additive configurations to interact multiplicatively. A foundational modern result is the theorem of Bourgain, Katz, and Tao \cite{BKT}, which gives power-saving sum--product growth in prime fields away from the very small and very large regimes. For larger sets, character-sum and incidence methods lead to explicit expansion thresholds; Hart, Iosevich, and Solymosi \cite{HIS} are an early example, while Murphy and Petridis \cite{MurphyPetridisSurvey} survey the later family of $p^{2/3}$-scale expander results.

Products of differences form one of the most concrete test cases in this circle of questions. Hart, Iosevich, and Solymosi \cite[Theorem~1.4]{HIS} proved that $\abs A>Cq^{3/4}$ implies
\[
(A-A)(A-A)=\F_q.
\]
Bennett, Hart, Iosevich, Pakianathan, and Rudnev \cite[Corollary~1.8]{BennettEtAl} lowered the scale for positive-proportion expansion: if $\abs A>Cq^{2/3}$, then $\abs{(A-A)(A-A)}\gg q$. This $q^{2/3}$ threshold is the direct predecessor of the arbitrary-field problem considered here; see also the discussion in \cite[Section~1.2]{MurphyPetridis}.

In prime fields, where there are no proper subfields, substantially stronger thresholds are available. Petridis \cite{PetridisPrime} proved that $\abs A\gg p^{5/8}$ implies
\[
\abs{(A\pm A)(A\pm A)}>\frac p2.
\]
Murphy, Petridis, Roche--Newton, Rudnev, and Shkredov \cite[Theorem~27]{MPRRS} subsequently obtained the pinned statement that $\abs A>Cp^{3/5}$ implies
\[
\abs{(A-a)(A-b)}>\frac p2
\]
for suitable $a,b\in A$. Their proof uses incidence input specific to prime fields; as explained in \cite[Section~1.2]{MurphyPetridis}, this mechanism does not extend in the same form to arbitrary finite extensions.

The arbitrary-field problem is qualitatively different because proper subfields create genuine obstructions. If $G<\F_q$ is a proper subfield and $A=G$, then $(A-A)(A-A)=G$; more generally, large pieces of subfield cosets or low-dimensional vector spaces over subfields can prevent the prime-field expansion mechanisms from applying directly. Murphy and Petridis \cite[Theorem~A]{MurphyPetridis} overcame this obstruction and, for the first time in arbitrary finite fields, crossed the exponent $2/3$: for every fixed $\delta<1/13542$, sufficiently large $q$ and every $A\subseteq\F_q$ with $\abs A>q^{2/3-\delta}$ satisfy
\[
\abs{(A-A)(A-A)}>\frac q2.
\]
Their argument combines an energy dichotomy with a structural analysis of sets for which many dilates have large additive energy, and it isolates the residual subfield structure in a form that can be handled separately.

Our main result improves the explicit saving while retaining uniformity over all finite fields.

\begin{theorem}\label{thm:main}
There is an absolute constant $C>0$ such that, for every finite field $\F_q$ and every $A\subseteq\F_q$ satisfying
\[
\abs{A}\ge Cq^{375/563},
\]
one has
\[
\abs{(A-A)(A-A)}>\frac q2.
\]
Since
\[
\frac{375}{563}=\frac23-\frac1{1689},
\]
every fixed $\delta<1/1689$ is admissible in the asymptotic formulation $\abs A>q^{2/3-\delta}$.
\end{theorem}

The proof retains the large-/small-energy dichotomy of Murphy--Petridis but reorganizes the large-energy branch around one common Fourier spectrum. Put $n=\abs A$ and
\[
D_{\times}(A)
=
\#\{(a_1,\ldots,a_8)\in A^8:(a_1-a_2)(a_3-a_4)=(a_5-a_6)(a_7-a_8)\}.
\]
Unless $D_{\times}(A)$ is already controlled by its random main term, their Theorem~B \cite[Theorem~B]{MurphyPetridis} supplies many $\xi$ for which $E^+(A,\xi A)$ is large. Instead of extracting a separate physical-space structure for each dilation parameter, we place all these energies on the same Fourier spectrum.

This common spectrum has two complementary forms of structure. Its popular multiplicative ratios allow the strong branch of the Reiher--Schoen Balog--Szemer\'edi--Gowers argument \cite[Lemma~3.1]{ReiherSchoen} to produce a large subset with small quotient set. At the same time, every subset of the spectrum inherits a quantitative additive-energy lower bound. We then decompose that additive energy by difference multiplicity. On each dense level, the Reiher--Schoen relation lemma \cite[Lemma~2.1]{ReiherSchoen} converts the restricted difference graph into a set with controlled difference doubling, while retaining the multiplicity parameter rather than discarding it in a second black-box BSG step.

The inverse input on each level is Mohammadi's arbitrary-field difference--ratio estimate \cite[Lemma~1]{Mohammadi}, which refines the sum--ratio method of Roche--Newton \cite{RocheNewton}. For a subfield-avoiding set $B\subseteq\F_q^\ast$ of size $O(q^{1/2})$, with
\[
d=\frac{\abs{B-B}}{\abs B},
\qquad
r=\frac{\abs{B/B}}{\abs B},
\]
one has
\[
d^7r^4\gg \abs B
\qquad\text{or}\qquad
d^6r^5\gg \abs B.
\]
Inserted into the levelwise argument, these two alternatives yield the structural bounds $Q\ll K^{186}$ and $Q\ll K^{188}$, where $Q=q/n$; the second is the numerical bottleneck.

If the subfield-avoidance condition fails, the resulting concentration is localized rather than treated as an error term. A difference--ratio incidence estimate obtained from Vinh's point--line theorem \cite[Theorem~3]{Vinh} excludes a subfield of order $q^{1/2}$. The remaining subfield has order $q^{1/3}$, and Fourier orthogonality transfers the frequency concentration to a coset of a two-dimensional vector space over that subfield. The vector-space result of Murphy--Petridis \cite[Theorem~E]{MurphyPetridis} then completes the structured branch.

Thus the improvement comes from three compatible ingredients: a common Fourier spectrum for all energetic dilates, a levelwise treatment of additive energy that preserves difference multiplicities, and the sharper difference--ratio exponents available in Mohammadi's arbitrary-field estimate. Section~\ref{sec:prelim} records the external inputs and the incidence estimate. Section~\ref{sec:fourier} develops the common spectrum, the levelwise energy argument, and the subfield localization. Section~\ref{sec:mainproof} balances the large- and small-$D_{\times}$ alternatives and proves \cref{thm:main}.

\section{Preliminaries and difference--ratio estimates}\label{sec:prelim}

All implicit constants in $\ll,\gg$ are absolute unless a parameter is indicated. Throughout the proof of the main theorem we fix $\eta=1/16$ in the Mohammadi estimate; consequently every constant occurring below is absolute. We shall choose the global small parameter $\kappa$ only after all structural constants have been fixed, and the constant $C$ in \cref{thm:main} only after $\kappa$ has been fixed. For finite nonempty $X,Y\subseteq\F_q$, write
\[
E^+(X,Y)=\#\{x_1+y_1=x_2+y_2:x_i\in X,\ y_i\in Y\},
\]
with $E^+(X)=E^+(X,X)$, and similarly $E^{\times}(X,Y)$ and $E^{\times}(X)$. We use $X/Y=\{x/y:x\in X,y\in Y,y\ne0\}$.

Fix a nontrivial additive character $\psi$ of $\F_q$, and use the unnormalised Fourier transform
\[
\widehat f(t)=\sum_{x\in\F_q}f(x)\psi(-tx).
\]
Thus Parseval gives $\sum_t\abs{\widehat f(t)}^2=q\sum_x\abs{f(x)}^2$.

We record the two external structural inputs used below.

\begin{lemma}[Murphy--Petridis large-energy alternative]\label{lem:MP-B}
Let $A\subseteq\F_q$, $n=\abs A$, and $K\ge1$. Suppose $n\le q/(4K)$ and
\[
D_{\times}(A)\ge \frac{n^8}{q}+\frac{3qn^5}{K}.
\]
Then there is $X\subseteq\F_q^{\ast}$ such that
\begin{equation}\label{eq:MP-energy}
E^+(A,\xi A)\ge \frac{n^3}{K}
\qquad (\xi\in X),
\end{equation}
and
\begin{equation}\label{eq:Xlower}
\abs X\ge \frac{q}{2Kn}.
\end{equation}
\end{lemma}

\begin{proof}
\cite[Theorem~B]{MurphyPetridis} produces a set $X_0\subseteq\F_q$ satisfying \eqref{eq:MP-energy} and $\abs{X_0}\ge q/(Kn)$.  We discard $0$ if it occurs.  Since the hypothesis $n\le q/(4K)$ gives $q/(Kn)\ge4$, the resulting set $X=X_0\setminus\{0\}$ satisfies $\abs X\ge q/(Kn)-1\ge q/(2Kn)$.  This yields \eqref{eq:Xlower} while ensuring that every dilation parameter is nonzero.
\end{proof}

\begin{lemma}[Reiher--Schoen energy and relation lemmas]\label{lem:RS}
Let $Y$ be a finite subset of an abelian group and suppose $E(Y)\ge \abs Y^3/L$ with $L\ge1$. For each fixed $\eps\in(0,1/2)$ there is $Y'\subseteq Y$ with
\[
\abs{Y'}\ge (1-\eps)L^{-1/2}\abs Y,
\qquad
\abs{Y'-Y'}\ll_{\eps}L^4\abs{Y'}.
\]
Moreover, in the popular-difference branch of their proof, where the relevant differences satisfy $r_{Y-Y}(d)\ge L^{-1/2}\abs Y$ and carry a fixed positive proportion of the energy, one obtains the stronger bound
\[
\abs{Y'-Y'}\ll_{\eps}L^3\abs{Y'}.
\]
Finally, if $R\subseteq Y^2$ has density $\delta$, then their relation lemma produces $Y'\subseteq Y$ with $\abs{Y'}\gg \delta\abs Y$ such that, for every $(y_1,y_2)\in (Y')^2$, there are $\gg \delta^4\abs Y^2\abs{Y'}$ triples $(x,b,z)\in Y^3$ for which $(y_1,x),(b,x),(b,z),(y_2,z)\in R$.
\end{lemma}

\begin{proof}
The first assertion is \cite[Theorem 1.2]{ReiherSchoen}, the strong popular branch is \cite[Lemma 3.1]{ReiherSchoen}, and the relation statement is \cite[Lemma 2.1]{ReiherSchoen}. We shall use the quantitative consequence of the latter explicitly in the proof of \cref{prop:energy-form}.
\end{proof}

The next lemma is the arbitrary-field inverse estimate needed by the energy-level argument. It is a published consequence of Roche--Newton's ratio-set method. Mohammadi explicitly observes that the covering argument may be run with $A-A$ in place of $A+A$ and records the resulting difference--ratio bounds; see \cite[Lemma 1 and its proof]{Mohammadi}.

\begin{lemma}[Mohammadi difference--ratio estimate]\label{prop:DR}
Fix $0<\eta<1/8$. Let $B\subseteq\F_q^{\ast}$ satisfy $\abs B\ll q^{1/2}$ and suppose that, for every proper subfield $G<\F_q$ and every $c\in\F_q$,
\begin{equation}\label{eq:subfieldavoid}
\abs{B\cap cG}\le \max\left\{\abs G^{1/2},\eta\abs B\right\}.
\end{equation}
Put $N=\abs B$, $D=\abs{B-B}$ and $R=\abs{B/B}$. Then
\begin{equation}\label{eq:DR}
D^7R^4\gg_{\eta} N^{12}
\qquad\text{or}\qquad
D^6R^5\gg_{\eta} N^{12}.
\end{equation}
Equivalently, with $d=D/N$ and $r=R/N$,
\[
d^7r^4\gg_{\eta}N
\qquad\text{or}\qquad
d^6r^5\gg_{\eta}N.
\]
\end{lemma}

\begin{proof}
Mohammadi's Lemma 1, applied with the constant $C=1$ in its subfield hypothesis, states that
\[
\abs{B\pm B}^{7}\abs{B/B}^{4}\gg_{\eta}\abs B^{12}
\qquad\text{or}\qquad
\abs{B\pm B}^{6}\abs{B/B}^{5}\gg_{\eta}\abs B^{12}.
\]
The notation $B\pm B$ in that lemma covers both the sum-set and difference-set forms.  Taking the difference-set form gives \eqref{eq:DR}.
\end{proof}

\begin{remark}\label{rem:constants}
Because $\eta=1/16$ is fixed once and for all, the constants implicit in \cref{prop:DR} are absolute. The same is true for every subsequent use of \cref{lem:RS}, where the auxiliary parameters are fixed numerical constants. Thus the generic bound in \cref{cor:R188} may be written as $Q\le C_{\mathrm g}K^{188}$ for one absolute constant $C_{\mathrm g}$, while the constants in the subfield alternative are absolute as well.
\end{remark}

We also need a form of Vinh's incidence estimate adapted to differences and ratios.

\begin{lemma}[Difference--ratio incidence estimate]\label{lem:VDR}
Let $B\subseteq\F_Q^{\ast}$, $h=\abs B$, $D=\abs{B-B}$ and $R=\abs{B/B}$. Then
\begin{equation}\label{eq:VDR}
h^2\le \frac{DRh}{Q}+Q^{1/2}\sqrt{DR}.
\end{equation}
Consequently, if $D\le dh$ and $R\le rh$, then either
\begin{equation}\label{eq:VDR-alt1}
h\gg \frac{Q}{dr}
\end{equation}
or
\begin{equation}\label{eq:VDR-alt2}
h\ll Q^{1/2}(dr)^{1/2}.
\end{equation}
\end{lemma}

\begin{proof}
Take the point set $P=(B-B)\times(B/B)$ and the $h^2$ lines
\[
\ell_{a,b}:\quad y=\frac{x+a}{b},
\qquad a,b\in B.
\]
They are pairwise distinct. Every triple $(a,b,c)\in B^3$ gives the incidence
\[
(c-a,c/b)\in\ell_{a,b},
\]
so $I(P,\mathcal L)\ge h^3$. Vinh's point--line incidence theorem over arbitrary finite fields \cite[Theorem~3]{Vinh} yields
\[
h^3\le \frac{DRh^2}{Q}+Q^{1/2}h\sqrt{DR},
\]
which is \eqref{eq:VDR}. If the first term on the right is at least $h^2/2$, then \eqref{eq:VDR-alt1} follows; otherwise the second term is at least $h^2/2$, which gives \eqref{eq:VDR-alt2}.
\end{proof}

\section{Fourier structure in the large-energy case}\label{sec:fourier}

Throughout this section let $A\subseteq\F_q$, $n=\abs A$, and
\[
Q=\frac qn.
\]
Assume that \cref{lem:MP-B} applies with a parameter $K\ge1$ and a set $X$ satisfying \eqref{eq:MP-energy}--\eqref{eq:Xlower}.

\subsection*{The common spectrum and the first BSG step}

Define
\[
F(t)=\abs{\widehat{\1_A}(t)}^2
\]
and the nonzero spectrum
\[
S=\left\{t\in\F_q^{\ast}:F(t)\ge \frac{n^2}{4K}\right\}.
\]

\begin{lemma}[Common spectrum]\label{lem:common-spectrum}
For every $\xi\in X$,
\begin{equation}\label{eq:overlap}
\abs{S\cap\xi^{-1}S}\ge \frac{Q}{4K},
\end{equation}
and
\begin{equation}\label{eq:Supper}
\abs S\le 4KQ.
\end{equation}
In particular, writing
\begin{equation}\label{eq:sdef}
\abs S=sQ,
\end{equation}
we have
\begin{equation}\label{eq:srange}
\frac1{4K}\le s\le4K.
\end{equation}
\end{lemma}

\begin{proof}
Fourier inversion gives
\begin{equation}\label{eq:fourier-energy}
E^+(A,\xi A)=\frac1q\sum_{t\in\F_q}F(t)F(\xi t).
\end{equation}
The contribution of $t=0$ is $n^4/q\le n^3/(4K)$ because $n\le q/(4K)$. On the set $t\ne0$, $t\notin S$, we have $F(t)<n^2/(4K)$, so Parseval gives
\[
\frac1q\sum_{\substack{t\ne0\\t\notin S}}F(t)F(\xi t)
\le
\frac{n^2}{4Kq}\sum_tF(\xi t)
=
\frac{n^3}{4K}.
\]
The same bound holds for the contribution from $\xi t\notin S$. In view of \eqref{eq:MP-energy} and \eqref{eq:fourier-energy}, at least $n^3/(4K)$ remains on $S\cap\xi^{-1}S$. Since every summand is at most $n^4$, this proves \eqref{eq:overlap}. Parseval also gives
\[
\abs S\frac{n^2}{4K}\le \sum_tF(t)=qn,
\]
which is \eqref{eq:Supper}. The lower bound in \eqref{eq:srange} follows from \eqref{eq:overlap} because $X$ is nonempty.
\end{proof}

The ratios in $X$ are not merely present in $S/S$; they are all popular. We exploit this before any additive structure extraction.

\begin{lemma}[Popular-ratio extraction]\label{lem:firstBSG}
There is $T\subseteq S$ such that, with $N=\abs T$,
\begin{equation}\label{eq:Tsize}
N\gg QK^{-3/2}s^{-1/2}
\end{equation}
and
\begin{equation}\label{eq:Tquot}
\abs{T/T}\ll K^9s^9N.
\end{equation}
\end{lemma}

\begin{proof}
For $\xi\in X$, \eqref{eq:overlap} says
\[
r_{S/S}(\xi)\ge \frac{Q}{4K}.
\]
Set
\[
L_{\times}=64K^3s^3.
\]
Using \eqref{eq:srange},
\[
L_{\times}^{-1/2}\abs S
=
\frac{Q}{8K^{3/2}s^{1/2}}
\le
\frac{Q}{4K}.
\]
Thus every $\xi\in X$ lies in the popular-ratio set at threshold $L_{\times}^{-1/2}\abs S$. Moreover,
\[
\sum_{\xi\in X}r_{S/S}(\xi)^2
\ge
\frac{Q}{2K}\left(\frac{Q}{4K}\right)^2
=
\frac{Q^3}{32K^3},
\]
while $\abs S^3/L_{\times}=Q^3/(64K^3)$.  Thus the popular ratios carry more than the fixed positive proportion of energy required by \cite[Lemma 3.1]{ReiherSchoen}. Applying that lemma in the multiplicative group $\F_q^{\ast}$ gives
\[
\abs T\gg L_{\times}^{-1/2}\abs S
\gg QK^{-3/2}s^{-1/2}
\]
and
\[
\abs{T/T}\ll L_{\times}^3\abs T\ll K^9s^9\abs T.
\]
\end{proof}

The Fourier spectrum also gives additive energy to every subset.

\begin{lemma}[Hereditary additive energy]\label{lem:hereditary}
For every $Y\subseteq S$,
\begin{equation}\label{eq:hered}
E^+(Y)\ge \frac{1}{16K^2Q}\abs Y^4.
\end{equation}
In particular, for $T$ in \cref{lem:firstBSG},
\begin{equation}\label{eq:Lplus}
E^+(T)\gg \frac{N^3}{L},
\qquad
L\ll K^{7/2}s^{1/2}.
\end{equation}
\end{lemma}

\begin{proof}
For $t\in Y$, choose a complex number $c_t$ of modulus one such that $c_t\widehat{\1_A}(t)=\abs{\widehat{\1_A}(t)}$. Since $Y\subseteq S$,
\[
\frac{n}{2\sqrt K}\abs Y
\le
\left|\sum_{a\in A}\sum_{t\in Y}c_t\psi(-ta)\right|.
\]
H\"older's inequality gives
\[
\frac{n}{2\sqrt K}\abs Y
\le
n^{3/4}
\left(\sum_{x\in\F_q}\left|\sum_{t\in Y}c_t\psi(-tx)\right|^4\right)^{1/4}.
\]
By character orthogonality, the fourth moment is at most $qE^+(Y)$, because for every additive quadruple the product of the four phases has modulus one. Raising to the fourth power yields
\[
E^+(Y)\ge \frac{n}{16K^2q}\abs Y^4
=
\frac{1}{16K^2Q}\abs Y^4,
\]
which proves \eqref{eq:hered}. Taking $Y=T$ and using \eqref{eq:Tsize},
\[
\frac{1}{16K^2Q}N^4
\gg
\frac{N^3}{K^{7/2}s^{1/2}},
\]
which is \eqref{eq:Lplus}.
\end{proof}

\subsection*{Levelwise additive structure}

The next proposition is the quantitative heart of the proof. It avoids the loss coming from immediately replacing the additive energy of $T$ by a single subset with $L^4$ difference doubling.

\begin{proposition}[Levelwise additive-structure alternative]\label{prop:energy-form}
Let $T\subseteq\F_q^{\ast}$, $N=\abs T$, and suppose $N\ll q^{1/2}$ and
\begin{equation}\label{eq:energy-input}
E^+(T)\gg \frac{N^3}{L},
\qquad
\abs{T/T}\le \rho N
\end{equation}
with $L,\rho\ge1$. Fix $\eta=1/16$. Then one of the following holds:
\begin{enumerate}[label=\textup{(\roman*)},leftmargin=2.2em]
\item there is $U\subseteq T$ with
\begin{equation}\label{eq:Ustructuredsize}
\abs U\gg \frac{N}{\sqrt L}
\end{equation}
and a proper subfield $G<\F_q$, $c\in\F_q^{\ast}$, such that
\begin{equation}\label{eq:Uconcentration}
\abs{U\cap cG}>\max\left\{\abs G^{1/2},\eta\abs U\right\};
\end{equation}
\item
\begin{equation}\label{eq:generic1}
N\ll \rho^4L^{28};
\end{equation}
\item
\begin{equation}\label{eq:generic2}
N\ll \rho^5L^{24}.
\end{equation}
\end{enumerate}
\end{proposition}

\begin{proof}
After increasing $L$ by an absolute factor, we may and shall assume that
\[
E^+(T)\ge \frac{N^3}{L}.
\]
This changes none of the asserted bounds except their absolute implicit constants. Write $m(d)=r_{T-T}(d)$.

\medskip
\noindent\emph{Popular differences.}
First consider
\[
P=\left\{d:m(d)\ge \frac{N}{\sqrt L}\right\}.
\]
If they carry a fixed positive proportion of the lower bound in \eqref{eq:energy-input}, the popular branch of \cref{lem:RS} gives $U\subseteq T$ with $\abs U\gg N/\sqrt L$ and
\[
\frac{\abs{U-U}}{\abs U}\ll L^3.
\]
Also $\abs{U/U}\le\rho N\ll \rho\sqrt L\abs U$. If $\abs U$ is bounded, the conclusions below are immediate after enlarging the absolute constants. Otherwise, if \eqref{eq:Uconcentration} fails for this $U$, \cref{prop:DR} gives either
\[
\abs U\ll L^{21}(\rho\sqrt L)^4\ll \rho^4L^{23}
\]
or
\[
\abs U\ll L^{18}(\rho\sqrt L)^5\ll \rho^5L^{41/2}.
\]
Since $\abs U\gg N/\sqrt L$, these imply, respectively, bounds stronger than \eqref{eq:generic1} and \eqref{eq:generic2}. Thus it remains to treat the weak differences
\[
W=\left\{d:m(d)<\frac{N}{\sqrt L}\right\},
\]
which we may assume carry $\gg N^3/L$ energy.

\medskip
\noindent\emph{Weak levels and the density threshold.}
Decompose $W$ into dyadic levels. For one level write
\[
D_{\mu}=\{d:\mu\le m(d)<2\mu\},
\qquad J=\abs{D_{\mu}},
\]
and put
\begin{equation}\label{eq:tz}
t=\frac{\mu}{N},
\qquad
z=\frac{J\mu^2}{N^3}.
\end{equation}
The relation
\[
\Gamma_{\mu}=\{(x,y)\in T^2:x-y\in D_{\mu}\}
\]
has density
\begin{equation}\label{eq:density}
\delta_{\mu}=\frac{\abs{\Gamma_{\mu}}}{N^2}\asymp \frac{J\mu}{N^2}=\frac zt.
\end{equation}
Since $t\le L^{-1/2}$ on weak levels and the dyadic values of $t$ form a geometric sequence, the total energy of the levels with $\delta_{\mu}<cL^{-1/2}$ is
\[
\sum_{\delta_{\mu}<cL^{-1/2}}zN^3
\ll
\frac{cN^3}{\sqrt L}\sum_t t
\ll
\frac{cN^3}{L}.
\]
Taking $c>0$ sufficiently small, these levels cannot carry all of the weak energy. We may therefore sum only over levels satisfying
\begin{equation}\label{eq:highdensity}
\delta_{\mu}\gg L^{-1/2}.
\end{equation}

\medskip
\noindent\emph{From a dense relation to controlled difference doubling.}
Apply \cite[Lemma 2.1]{ReiherSchoen} to $\Gamma_{\mu}$ with a fixed auxiliary parameter. It yields $U_{\mu}\subseteq T$ with
\begin{equation}\label{eq:Umusize}
\abs{U_{\mu}}\gg \delta_{\mu}N.
\end{equation}
For every ordered pair $(u_1,u_2)\in U_{\mu}^2$, the lemma gives $\gg \delta_{\mu}^4N^2\abs{U_{\mu}}$ triples $(x,b,y)\in T^3$ such that
\[
(u_1,x),(b,x),(b,y),(u_2,y)\in\Gamma_{\mu}.
\]
For each such triple,
\[
u_1-u_2=(u_1-x)-(b-x)+(b-y)-(u_2-y),
\]
and all four displayed differences lie in $D_{\mu}$. For fixed $u_1,u_2$, the resulting quadruple of differences determines $(x,b,y)$ uniquely. Hence every element of $U_{\mu}-U_{\mu}$ has at least $\gg \delta_{\mu}^4N^2\abs{U_{\mu}}$ representations as $d_1-d_2+d_3-d_4$ with $d_i\in D_{\mu}$. Since there are at most $J^4$ such quadruples,
\[
\abs{U_{\mu}-U_{\mu}}
\ll
\frac{J^4}{\delta_{\mu}^4N^2\abs{U_{\mu}}}.
\]
Dividing by $\abs{U_{\mu}}$ and using \eqref{eq:Umusize} gives
\[
\frac{\abs{U_{\mu}-U_{\mu}}}{\abs{U_{\mu}}}
\ll
\frac{J^4}{\delta_{\mu}^6N^4}.
\]
Substituting \eqref{eq:tz}--\eqref{eq:density}, we obtain
\begin{equation}\label{eq:dmu}
d_{\mu}:=\frac{\abs{U_{\mu}-U_{\mu}}}{\abs{U_{\mu}}}
\ll z^{-2}t^{-2}.
\end{equation}
Since $U_{\mu}/U_{\mu}\subseteq T/T$, \eqref{eq:Umusize} also gives
\begin{equation}\label{eq:rmu}
r_{\mu}:=\frac{\abs{U_{\mu}/U_{\mu}}}{\abs{U_{\mu}}}
\ll \rho\frac tz.
\end{equation}

\medskip
\noindent\emph{Difference--ratio alternatives on each level.}
If some high-density level satisfies \eqref{eq:Uconcentration}, then \eqref{eq:highdensity} and \eqref{eq:Umusize} give \eqref{eq:Ustructuredsize}, and we are in alternative (i). Assume instead that every high-density level is subfield-avoiding. If $\abs{U_{\mu}}$ is bounded by a sufficiently large absolute constant, then \eqref{eq:highdensity} and \eqref{eq:Umusize} imply $N\ll L^{1/2}$, which is already stronger than \eqref{eq:generic1}; hence such a level may be discarded from the remaining argument. We may therefore apply \cref{prop:DR} to every remaining $U_{\mu}$. Using \eqref{eq:Umusize}, \eqref{eq:dmu}, and \eqref{eq:rmu}, the first branch gives
\[
\rho^4z^{-18}t^{-10}
\gg
Nz t^{-1},
\]
so
\begin{equation}\label{eq:z1}
z\ll \left(\frac{\rho^4}{N}\right)^{1/19}t^{-9/19}.
\end{equation}
The second branch gives
\begin{equation}\label{eq:z2}
z\ll \left(\frac{\rho^5}{N}\right)^{1/18}t^{-1/3}.
\end{equation}
Finally, $J\mu\le N^2$ implies the trivial estimate
\begin{equation}\label{eq:ztrivial}
z\ll t.
\end{equation}

\medskip
\noindent\emph{Summation over levels.}
Partition the high-density levels according to which alternative in \cref{prop:DR} holds. For the levels in the first class, the envelope formed by \eqref{eq:z1} and \eqref{eq:ztrivial} satisfies
\begin{equation}\label{eq:envelope1}
\sum_t\min\left\{t,
\left(\frac{\rho^4}{N}\right)^{1/19}t^{-9/19}\right\}
\ll
\left(\frac{\rho^4}{N}\right)^{1/28}.
\end{equation}
Indeed, split the geometric sum at the point where the two terms are equal. For the levels in the second class, similarly,
\begin{equation}\label{eq:envelope2}
\sum_t\min\left\{t,
\left(\frac{\rho^5}{N}\right)^{1/18}t^{-1/3}\right\}
\ll
\left(\frac{\rho^5}{N}\right)^{1/24}.
\end{equation}
Since the retained weak levels carry $\gg N^3/L$ additive energy, their total normalized contribution is $\gg1/L$. Hence \eqref{eq:envelope1}--\eqref{eq:envelope2} imply
\[
\frac1L
\ll
\left(\frac{\rho^4}{N}\right)^{1/28}
+
\left(\frac{\rho^5}{N}\right)^{1/24}.
\]
At least one summand is $\gg1/L$, yielding \eqref{eq:generic1} or \eqref{eq:generic2}.
\end{proof}

We now substitute the Fourier parameters.

\begin{corollary}[Structural exponent $188$]\label{cor:R188}
Under the hypotheses of \cref{lem:MP-B}, and assuming additionally $KQ\ll q^{1/2}$, one of the following holds:
\begin{enumerate}[label=\textup{(\roman*)},leftmargin=2.2em]
\item $Q\le C_{\mathrm g}K^{188}$ for an absolute constant $C_{\mathrm g}$;
\item there are $U\subseteq T$, a proper subfield $G<\F_q$, and $c\in\F_q^{\ast}$ such that
\begin{equation}\label{eq:Ulower-final}
\abs U\gg QK^{-13/4}s^{-3/4}\gg QK^{-4}
\end{equation}
and \eqref{eq:Uconcentration} holds.
\end{enumerate}
\end{corollary}

\begin{proof}
For $T$ from \cref{lem:firstBSG}, take
\[
\rho\ll K^9s^9,
\qquad
L\ll K^{7/2}s^{1/2}
\]
from \eqref{eq:Tquot} and \eqref{eq:Lplus}. Since $T\subseteq S$ and $\abs S\ll KQ\ll q^{1/2}$ by hypothesis, the size condition in \cref{prop:energy-form} is satisfied. If \cref{prop:energy-form}(i) holds, then \eqref{eq:Tsize} and \eqref{eq:Ustructuredsize} give \eqref{eq:Ulower-final}. Otherwise, \eqref{eq:generic1} and \eqref{eq:Tsize} imply
\[
QK^{-3/2}s^{-1/2}
\ll
(K^9s^9)^4(K^{7/2}s^{1/2})^{28}
=
K^{134}s^{50},
\]
whence
\[
Q\ll K^{271/2}s^{101/2}\ll K^{186}
\]
by \eqref{eq:srange}. Likewise \eqref{eq:generic2} yields
\[
Q\ll K^{261/2}s^{115/2}\ll K^{188}.
\]
Thus the generic case gives $Q\le C_{\mathrm g}K^{188}$ for an absolute constant $C_{\mathrm g}$, as asserted.
\end{proof}

\subsection*{The subfield alternative}

We show that, in the range relevant to \cref{thm:main}, alternative (ii) in \cref{cor:R188} leads to a $q^{1/3}$ subfield and then back to a large subset of a two-dimensional vector space.

\begin{lemma}[Quadratic-subfield exclusion]\label{lem:halfexclude}
Assume
\begin{equation}\label{eq:critical-range}
Cq^{375/563}\le n\le q^{2/3},
\qquad
K\le c q^{1/563}
\end{equation}
with $C$ sufficiently large and $c$ sufficiently small. Suppose $U\subseteq T$ satisfies \eqref{eq:Ulower-final} and \eqref{eq:Uconcentration}. Then the subfield $G$ in \eqref{eq:Uconcentration} cannot have cardinality $q^{1/2}$.
\end{lemma}

\begin{proof}
Let $H=U\cap cG$. Then $\abs H>\eta\abs U\gg QK^{-4}$, where $\eta=1/16$. Since $H\subseteq T\subseteq S$, \cref{lem:hereditary} gives
\[
E^+(H)\ge \frac{\abs H^4}{16K^2Q}\gg \frac{\abs H^3}{K^6}.
\]
Applying \cref{lem:RS} additively, with a fixed $\eps$, gives $C_0\subseteq H$ satisfying
\begin{equation}\label{eq:C0size}
\abs{C_0}\gg QK^{-7}
\end{equation}
and
\begin{equation}\label{eq:C0diff}
\abs{C_0-C_0}\ll K^{24}\abs{C_0}.
\end{equation}
Also $C_0/C_0\subseteq T/T$. By \eqref{eq:Tquot}, $N\le\abs S=sQ$, \eqref{eq:srange}, and \eqref{eq:C0size},
\begin{equation}\label{eq:C0ratio}
\frac{\abs{C_0/C_0}}{\abs{C_0}}
\ll K^{16}s^{10}
\ll K^{26}.
\end{equation}
After dilating by $c^{-1}$ we may regard $C_0$ as a subset of $G$ without changing the difference- or ratio-doubling constants. Apply \cref{lem:VDR} inside $G$. If $\abs G=q^{1/2}$, \eqref{eq:C0diff}--\eqref{eq:C0ratio} give either
\begin{equation}\label{eq:halfcase1}
\abs{C_0}\gg q^{1/2}K^{-50}
\end{equation}
or
\begin{equation}\label{eq:halfcase2}
\abs{C_0}\ll q^{1/4}K^{25}.
\end{equation}
In the first case, because $C_0\subseteq S$ and $\abs S\ll KQ=Kq/n$,
\[
n\ll K^{51}q^{1/2},
\]
contradicting \eqref{eq:critical-range}: the exponent on the right is at most $1/2+51/563<375/563$. In the second case, \eqref{eq:C0size} implies
\[
\frac{q^{3/4}}{n}\ll K^{32}.
\]
But $n\le q^{2/3}$ in \eqref{eq:critical-range}, so the left-hand side is at least $q^{1/12}$, whereas $K^{32}\le c^{32}q^{32/563}$ and $32/563<1/12$. This is impossible for sufficiently small $c$ (and, after enlarging the absolute constants, for all remaining $q$).
\end{proof}

\begin{lemma}[Cubic-subfield localization]\label{lem:cubic}
Under \eqref{eq:critical-range}, if alternative (ii) of \cref{cor:R188} holds, then there is a subfield $G\le\F_q$ with $\abs G=q^{1/3}$, an additive subgroup $V\le\F_q$ which is a two-dimensional vector space over $G$, and $a\in\F_q$ such that
\begin{equation}\label{eq:cosetmass}
\abs{A\cap(a+V)}\gg q^{2/3}K^{-5}.
\end{equation}
\end{lemma}

\begin{proof}
By \eqref{eq:Ulower-final}, \eqref{eq:critical-range}, and $Q=q/n$, the upper bound $n\le q^{2/3}$ gives
\[
\abs U\gg QK^{-4}\gg q^{1/3-4/563}>q^{1/4},
\]
while $U\subseteq S$ and \eqref{eq:Supper} give
\[
\abs U\ll KQ\ll q^{189/563}.
\]
Thus $q^{1/4}<\abs U<q^{1/2}$ after adjusting the constants. The subfield $G$ in \eqref{eq:Uconcentration} is proper by construction. Since $\abs{U\cap cG}>\eta\abs U$ with $\eta=1/16$, we have $\abs G\gg\abs U>q^{1/4}$. Every proper subfield of $\F_q$ has order $q^{1/r}$ for an integer extension degree $r\ge2$. The condition $\abs G>q^{1/4}$ leaves only $r=2$ or $r=3$. The case $r=2$ is excluded by \cref{lem:halfexclude}, so $\abs G=q^{1/3}$.

Set $W=cG$, viewed as an additive subgroup of $\F_q$, and let
\[
V=W^{\perp}=\{x\in\F_q:\psi(wx)=1\text{ for all }w\in W\}.
\]
Then $\abs V=q/\abs W=q^{2/3}$. Moreover, $V$ is closed under multiplication by $G$: if $g\in G$, $x\in V$, and $w=cg'\in W$, then $wgx=c(gg')x$, and $c(gg')\in W$. Hence $V$ is a two-dimensional vector space over $G$.

Let $H=U\cap W$. By \eqref{eq:Uconcentration} and \eqref{eq:Ulower-final}, $\abs H\gg QK^{-4}$. Character orthogonality gives
\begin{equation}\label{eq:annihilator}
\sum_{t\in W}\abs{\widehat{\1_A}(t)}^2
=
\abs W\sum_{C\in\F_q/V}\abs{A\cap C}^2.
\end{equation}
Since $H\subseteq U\subseteq T\subseteq S$,
\[
\sum_{t\in W}\abs{\widehat{\1_A}(t)}^2
\ge
\frac{\abs H n^2}{4K}.
\]
Combining this with \eqref{eq:annihilator} and $\sum_C\abs{A\cap C}=n$, there is a coset $a+V$ for which
\[
\abs{A\cap(a+V)}
\ge
\frac{\abs H n}{4K\abs W}
\gg
\frac{q}{K^5\abs G}
=
q^{2/3}K^{-5}.
\]
This is \eqref{eq:cosetmass}.
\end{proof}

\begin{proposition}[Large-energy branch]\label{prop:largebranch}
There exist absolute constants $c_{\ast},C_{\ast}>0$ and $q_{\ast}$ such that the following holds. Let $q\ge q_{\ast}$, let $A\subseteq\F_q$, put $n=\abs A$ and $Q=q/n$, and suppose
\[
C_{\ast}q^{375/563}\le n\le q^{2/3},
\qquad
1\le K\le c_{\ast}Q^{1/188}.
\]
If the large-energy hypothesis of \cref{lem:MP-B} holds, then
\[
\abs{(A-A)(A-A)}>\frac q2.
\]
\end{proposition}

\begin{proof}
Let $C_{\mathrm g}$ be the absolute constant in \cref{cor:R188}. Let $c_{\mathrm h}>0$ and $C_{\mathrm h}>0$ be constants for which \cref{lem:halfexclude} is valid, after increasing a lower threshold for $q$ if necessary. Choose $c_{\ast}>0$ so small that
\begin{equation}\label{eq:cstar-choice}
C_{\mathrm g}c_{\ast}^{188}\le \frac12,
\qquad
c_{\ast}\le c_{\mathrm h},
\qquad
c_{\ast}\le1.
\end{equation}
We then choose $C_{\ast}\ge \max\{1,C_{\mathrm h}\}$ and finally enlarge $q_{\ast}$ so that all fixed-constant asymptotic inequalities used below hold for $q\ge q_{\ast}$.

Since $n\ge C_{\ast}q^{375/563}$,
\[
Q=\frac qn\le C_{\ast}^{-1}q^{188/563}.
\]
Hence
\[
KQ\le c_{\ast}Q^{189/188}
\ll q^{189/563}=o(q^{1/2}),
\]
so the size hypothesis in \cref{cor:R188} holds once $q\ge q_{\ast}$. If the generic alternative of that corollary held, then by $K\le c_{\ast}Q^{1/188}$ and \eqref{eq:cstar-choice},
\[
Q\le C_{\mathrm g}K^{188}
\le C_{\mathrm g}c_{\ast}^{188}Q
\le \frac Q2,
\]
a contradiction. Therefore the subfield alternative holds.

Moreover,
\[
K\le c_{\ast}Q^{1/188}
\le c_{\ast}q^{1/563}
\le c_{\mathrm h}q^{1/563},
\]
and $C_{\ast}\ge C_{\mathrm h}$, so \cref{lem:cubic} applies. We obtain a subfield $G$ with $\abs G=q^{1/3}$, a two-dimensional $G$-vector space $V$, and $a\in\F_q$ such that
\[
\abs{A\cap(a+V)}\ge c_0 q^{2/3}K^{-5}
\]
for an absolute constant $c_0>0$.

Set $A_0=(A-a)\cap V$. Then $A_0-A_0\subseteq A-A$, and
\[
\abs{A_0}
\ge c_0q^{2/3}K^{-5}
\ge c_0c_{\ast}^{-5}q^{2/3-5/563}.
\]
Since $2/3-5/563>7/12$, after increasing $q_{\ast}$ we have $\abs{A_0}\ge \sqrt2\,q^{7/12}$. By \cite[Theorem~E]{MurphyPetridis},
\[
(A_0-A_0)(A_0-A_0)=VV
\]
and $\abs{VV}>q/2$. Since $(A_0-A_0)(A_0-A_0)\subseteq(A-A)(A-A)$, the proposition follows.
\end{proof}

\section{Proof of the main theorem}\label{sec:mainproof}

\begin{proof}[Proof of \cref{thm:main}]
Let $c_{\ast},C_{\ast},q_{\ast}$ be the constants from \cref{prop:largebranch}. Choose once and for all
\begin{equation}\label{eq:kappa-choice}
0<\kappa\le \min\left\{\frac14,\frac{c_{\ast}}2\right\}.
\end{equation}
Next choose $C\ge C_{\ast}$ so large that
\begin{equation}\label{eq:C-choice}
\kappa C^{563/188}\ge6
\end{equation}
and, if necessary, also
\begin{equation}\label{eq:smallq-vacuous}
Cq^{375/563}>q
\qquad (q<q_{\ast}).
\end{equation}
The latter requirement involves only finitely many $q$, so it can be enforced by increasing $C$. This fixes the constants in the theorem in the order
\[
\text{structural constants}\;\longrightarrow\;\kappa\;\longrightarrow\;C.
\]

Let $A\subseteq\F_q$ satisfy $n:=\abs A\ge Cq^{375/563}$ and put $Q=q/n$. By \eqref{eq:smallq-vacuous}, the hypothesis is impossible when $q<q_{\ast}$, so assume $q\ge q_{\ast}$. The case $n>q^{2/3}$ follows directly from \cite[Theorem~A]{MurphyPetridis}, after enlarging the fixed threshold $q_{\ast}$ if necessary. Hence we may suppose
\begin{equation}\label{eq:main-range}
Cq^{375/563}\le n\le q^{2/3}.
\end{equation}
Define
\begin{equation}\label{eq:Kchoice}
K=\kappa Q^{1/188}.
\end{equation}
Since $n\le q^{2/3}$, we have $Q\ge q^{1/3}$. After enlarging $q_{\ast}$ once more, this implies $K\ge1$. Also $Q\ge1$ and \eqref{eq:kappa-choice} give
\[
4K=4\kappa Q^{1/188}\le Q^{1/188}\le Q,
\]
so $n=q/Q\le q/(4K)$, as required in \cref{lem:MP-B}. Finally, \eqref{eq:kappa-choice} gives $K\le c_{\ast}Q^{1/188}$.

If
\[
D_{\times}(A)\ge \frac{n^8}{q}+\frac{3qn^5}{K},
\]
then all hypotheses of \cref{prop:largebranch} follow from \eqref{eq:main-range}, and therefore
\[
\abs{(A-A)(A-A)}>\frac q2.
\]

It remains to consider the complementary case
\begin{equation}\label{eq:smallD}
D_{\times}(A)< \frac{n^8}{q}+\frac{3qn^5}{K}.
\end{equation}
By Cauchy--Schwarz,
\begin{equation}\label{eq:CSproduct}
\abs{(A-A)(A-A)}\,D_{\times}(A)\ge n^8.
\end{equation}
We claim that our choice of $C$ guarantees
\begin{equation}\label{eq:precise-smallcondition}
K\ge 6\frac{q^2}{n^3}.
\end{equation}
Indeed, using \eqref{eq:Kchoice},
\[
\frac{K}{q^2/n^3}
=
\kappa\frac{n^{3-1/188}}{q^{2-1/188}}.
\]
The right-hand side is increasing in $n$, and at the lower endpoint $n=Cq^{375/563}$ of \eqref{eq:main-range} the powers of $q$ cancel because
\[
\frac{375}{563}=\frac{2-1/188}{3-1/188}.
\]
Consequently
\[
\frac{K}{q^2/n^3}
\ge
\kappa C^{3-1/188}
=
\kappa C^{563/188}
\ge6
\]
by \eqref{eq:C-choice}, proving \eqref{eq:precise-smallcondition}.

It follows that
\[
\frac{3qn^5}{K}
\le
\frac12\frac{n^8}{q}.
\]
Thus \eqref{eq:smallD} yields
\[
D_{\times}(A)<\frac32\frac{n^8}{q}.
\]
Combining this with \eqref{eq:CSproduct} gives the slightly stronger bound
\[
\abs{(A-A)(A-A)}>\frac{2q}{3}>\frac q2.
\]
This completes the proof.
\end{proof}

\paragraph{AI Assistance Statement}

The author used ChatGPT, with GPT-5.6 available through a ChatGPT Plus subscription, as an auxiliary research and writing tool during parts of the exploratory development of this work. The tool was used for brainstorming possible proof strategies, checking intermediate calculations and logical dependencies, organizing arguments, and assisting with the exposition and preparation of the manuscript.

All mathematical statements, proofs, references, and final formulations appearing in this paper were independently reviewed and selected by the author. The author assumes full responsibility for the correctness, originality, and presentation of the results.


\begin{thebibliography}{99}

\bibitem{BennettEtAl}
M. Bennett, D. Hart, A. Iosevich, J. Pakianathan, and M. Rudnev,
\emph{Group actions and geometric combinatorics in $\F_q^d$},
Forum Math. \textbf{29} (2017), no.~1, 91--110, doi:10.1515/forum-2015-0251.

\bibitem{BKT}
J. Bourgain, N. Katz, and T. Tao,
\emph{A sum-product estimate in finite fields, and applications},
Geom. Funct. Anal. \textbf{14} (2004), no.~1, 27--57, doi:10.1007/s00039-004-0451-1.

\bibitem{HIS}
D. Hart, A. Iosevich, and J. Solymosi,
\emph{Sum-product estimates in finite fields via Kloosterman sums},
Int. Math. Res. Not. IMRN (2007), Art. ID rnm007, 14 pp., doi:10.1093/imrn/rnm007.

\bibitem{Mohammadi}
A. Mohammadi,
\emph{Szemer\'edi--Trotter type results in arbitrary finite fields},
Integers \textbf{20} (2020), Paper No. A7, 29 pp., doi:10.5281/zenodo.10717435.

\bibitem{MPRRS}
B. Murphy, G. Petridis, O. Roche--Newton, M. Rudnev, and I. D. Shkredov,
\emph{New results on sum--product type growth over fields},
Mathematika \textbf{65} (2019), no.~3, 588--642, doi:10.1112/S0025579319000044.

\bibitem{MurphyPetridisSurvey}
B. Murphy and G. Petridis,
\emph{A second wave of expanders over finite fields},
in \emph{Combinatorial and Additive Number Theory II}, Springer Proc. Math. Stat. \textbf{220}, Springer, Cham, 2018, 215--238, doi:10.1007/978-3-319-68032-3\_15.

\bibitem{MurphyPetridis}
B. Murphy and G. Petridis,
\emph{Products of differences over arbitrary finite fields},
Discrete Analysis (2018), Paper No.~18, 42 pp., doi:10.19086/da.5098.

\bibitem{PetridisPrime}
G. Petridis,
\emph{Products of differences in prime order finite fields},
preprint (2016), arXiv:1602.02142.

\bibitem{ReiherSchoen}
C. Reiher and T. Schoen,
\emph{Note on the theorem of Balog, Szemer\'edi, and Gowers},
Combinatorica \textbf{44} (2024), 691--698, doi:10.1007/s00493-024-00092-5.

\bibitem{RocheNewton}
O. Roche--Newton,
\emph{Sum-ratio estimates over arbitrary finite fields},
preprint (2014), arXiv:1407.1654.

\bibitem{Vinh}
L. A. Vinh,
\emph{The Szemer\'edi--Trotter type theorem and the sum-product estimate in finite fields},
European J. Combin. \textbf{32} (2011), no.~8, 1177--1181, doi:10.1016/j.ejc.2011.06.008.

\end{thebibliography}
\end{document}